\documentclass{amsart}
\usepackage{graphicx} 
\usepackage{amsfonts}
\usepackage{amsmath}
\usepackage{amsthm,color}
\usepackage{tikz}
\usepackage{bbm}
\usepackage[hidelinks,plainpages=false,pdfpagelabels]{hyperref}
\usepackage{thmtools}
\usepackage[capitalise, noabbrev]{cleveref}
\usepackage[margin=1in]{geometry}
\newtheorem{thm}{Theorem}
\numberwithin{thm}{section}
\newtheorem{prop}[thm]{Proposition}
\newtheorem{lem}[thm]{Lemma}
\newtheorem{cor}[thm]{Corollary}

\theoremstyle{remark}
\newtheorem{rem}[thm]{Remark}
\theoremstyle{definition}
\newtheorem{defn}[thm]{Definition}
\newtheorem{eg}[thm]{Example}
\usepackage[normalem]{ulem}

\newtheorem*{lthm}{Main Theorem}

\newcommand{\Diag}{\mathrm{Diag}}
\newcommand{\Pic}{\mathrm{Pic}}
\newcommand{\Div}{\mathrm{Div}}
\newcommand{\im}{\mathrm{im}}
\newcommand{\ZZ}{\mathbb{Z}}
\definecolor{Green}{rgb}{0.0, 0.5, 0.0}

\hypersetup{
 colorlinks=true,
 linkcolor=blue,
 filecolor=blue,
 citecolor=olive,
 urlcolor=orange,
 }

\author[D.~Labib]{Daniel Labib}
\address[Labib]{Department of Mathematics and Statistics\\University of Ottawa\\
150 Louis-Pasteur Pvt\\
Ottawa, ON\\
Canada K1N 6N5}
\email{dlabi091@uottawa.ca}

\author[A.~Lei]{Antonio Lei}
\address[Lei]{Department of Mathematics and Statistics\\University of Ottawa\\
150 Louis-Pasteur Pvt\\
Ottawa, ON\\
Canada K1N 6N5}
\email{antonio.lei@uottawa.ca}

\title{Recovering Laplacian Lattices from $L$-Functions of Graphs}
\begin{document}

\maketitle

\begin{abstract}
    We introduce $L$-functions associated with characters of the Jacobian of a finite graph, as a graph-theoretic analogue of the $L$-functions arising from unramified coverings of algebraic curves. These $L$-functions are defined using the Riemann--Roch structure on the graph and extend Lorenzini's two-variable zeta function. We show that if two connected graphs without bridges have isomorphic Jacobians and their $L$-functions agree under the induced correspondence of characters, then their Laplacian lattices coincide. We also show that these $L$-functions satisfy a functional equation. Finally, we give examples showing that neither the Jacobian nor the Lorenzini zeta function alone determines the Laplacian lattice.
\end{abstract}
\section{Introduction}
Throughout this article, the term “graph” will refer exclusively to undirected, finite multigraphs without loops.
The Jacobian of a graph provides a natural combinatorial analogue of the Jacobian variety associated with an algebraic curve.
For a graph $X$, {the Picard group of $X$ is defined as the quotient of the group of divisors by the image of the graph Laplacian. The Jacobian of $X$, denoted by $J_X$, is defined as the degree-zero subgroup of the Picard group.
When $X$ is connected, $J_X$ is a finite abelian group,} whose cardinality coincides with the number of spanning trees of $X$. In \cite{BN07}, Baker and Norine established that the divisor theory on graphs satisfies a Riemann–Roch theorem and admits Abel–Jacobi maps, thereby yielding striking combinatorial counterparts to the classical constructions in algebraic geometry.
In \cite{zini12}, Lorenzini constructed a two-variable zeta function for connected graphs, which is of the form
\[
\zeta(t,u,X)=\frac{f(t,u)}{(1-t)(1-ut)}
\]
for some polynomial $f(t,u)\in\mathbb{Z}[t,u]$. The function encodes important arithmetic information about the graph. For example, $f(1,u)=|J_X|$. Let $V(X)$ be the set of vertices of $X$. The image of the Laplacian operator on $\ZZ^{V(X)}$ is called the \textit{\textbf{Laplacian lattice}}. An important feature of $\zeta(t,u,X)$ is that it depends only on the Laplacian lattice.

In \cite{BV}, Booher and Voloch proved that a smooth hyperbolic projective curve over a finite field can be recovered from the $L$-functions attached to the Hilbert class field of the curve and its constant field extensions. Motivated by the deep analogy between the arithmetic of curves and the theory of graphs, one is led to ask whether analogous phenomena also arise in the setting of graphs.

To address this question, we begin by introducing the $L$-function $L(t,u,X,\chi{,D_1})$ associated with a character $\chi$ of $J_X$ {and} a degree-one divisor class $D_1$. Our construction arises from a slight modification of Lorenzini’s zeta function and is directly analogous to the corresponding definitions for curves; see \cref{def:L}.
The trivial character recovers Lorenzini's zeta function, {while the nontrivial characters provide additional information that is invisible to the zeta function alone. Indeed, two graphs may have the same Lorenzini zeta function but distinct Laplacian lattices, as illustrated in \cref{rk:sameLorenzini}. On the other hand, the main result of this article shows that the full collection of $L$-functions, together with an isomorphism of the Jacobians, determines the Laplacian lattice. In particular, we show:}

\begin{lthm}[\cref{mainthm}]
    Let $X$ and $X'$ be connected graphs without bridges. Suppose there exists a group isomorphism $\phi:J_{X'}\to J_X$, and there exist $D_1\in \Pic^1(X)$ and $D_1'\in \Pic^1(X')$ such that for all characters $\chi$ of $J_X$, we have the equality
\[
L(t,u,X,\chi{,D_1}) = L(t,u,X',\chi\circ\phi{,D_1'}).
\]
There exists an isomorphism 
$\tilde\phi:\Pic(X')\to\Pic(X)$ extending $\phi$ such that $\tilde\phi(D_1')=D_1$. Furthermore, it
gives rise to a bijection $\hat\phi:V(X')\to V(X)$ and satisfies
\[
h_{X'}(D)=h_X\bigl(\tilde\phi(D)\bigr)
\qquad\text{for every }D\in\Pic(X').
\]
Under $\hat\phi$, the natural identification of
$\mathbb Z^{V(X')}$ with $\mathbb Z^{V(X)}$ maps the
Laplacian lattice of $X'$ bijectively onto that of $X$.
\end{lthm}

Note that if $X$ and $X'$ are connected bridgeless graphs with the same Lorenzini zeta function, then \cite[Lemma~3.14]{zini12} shows that $X$ and $X'$ have the same number of vertices and edges, and that their Jacobians have the same cardinality. Thus, \cref{thmA} may be viewed as a refinement of this result. However, equality of the Lorenzini zeta functions does not imply that
the Jacobians are isomorphic; see \cite[Example~3.15]{zini12}.
Thus, the isomorphism $\phi:J_{X'}\to J_X$ in Theorem~4.3
is an additional hypothesis.

The proof proceeds by exploiting the dependence of the $L$-functions on the Riemann--Roch structure. From the equality of $L$-functions, we can compare the values of the Riemann--Roch function on divisor classes of $X$ and $X'$. We show that this leads to a bijection of the effective degree-one classes. When the graphs are connected and bridgeless, the Abel--Jacobi map identifies the vertices with these degree-one classes. This, in turn, forces the two Laplacian lattices to coincide.
The hypothesis that the graphs have no bridges is natural from the perspective of Lorenzini's zeta function. In fact, contracting all bridges does not change the zeta function; see \cref{prop:contraction}.

We emphasize that neither the Jacobian nor the Lorenzini zeta function alone contains enough information to recover the Laplacian lattice. See \cref{rk:genuine} and \cref{rk:sameLorenzini} where explicit examples are given, illustrating that the additional information encoded in the $L$-functions is genuinely necessary to recover the Laplacian lattice.

In \cref{cor:iso}, we show that if we further assume that the graphs $X$ and $X'$ in \cref{thmA} are such that every pair of distinct vertices within the same graph is connected by at least one edge, then they are in fact isomorphic. However, there exist non-isomorphic graphs with the same Laplacian lattice (see \cref{eg:non-iso}). Manjunath \cite{manjunath} showed that the Delaunay triangulation associated with the Laplacian lattice contains complete information about the underlying graph up to isomorphism. This raises the natural question of what input in addition to the $L$-functions is required to determine a graph without bridges up to isomorphism.

Finally, we note that there is extensive literature on zeta and $L$-functions associated with graph coverings {(see \cite{Ihara})}. Our $L$-functions are of a different nature: they are defined using the Riemann--Roch structure on the Jacobian rather than via prime cycles and graph coverings. To the best of our knowledge, there is no obvious relation between these two types of $L$-function.

\subsection*{Organization}
The article is organized as follows. In \cref{sec:notions}, we recall the necessary notions from the divisor theory of graphs, including the Picard group, Jacobian, Riemann--Roch structure, and Abel--Jacobi map. In \cref{sec:functions}, we review Lorenzini's two-variable zeta function and introduce the $L$-functions associated with characters of the Jacobian. Finally, in \cref{sec:mainproof}, we prove the main theorem recovering the Laplacian lattice from the collection of $L$-functions.

\subsection*{Acknowledgement}
We thank Milo Knysh, Dino Lorenzini and Felipe Voloch for valuable discussions during the preparation of this article. We are also grateful to the referee for a thorough reading of an earlier version of the article and for numerous helpful comments and suggestions that resulted in substantial improvements. This research was supported by an NSERC Alliance International Grant (ALLRP 613854-25). DL also received an NSERC Undergraduate Student Research Award while this work was being carried out. 

\subsection*{Statement on the use of AI}
During the preparation of this article, AI has been used to check for typos, grammatical mistakes, and English phrasing. Some examples have been found with the help of Claude AI and Gemini. All mathematical statements and assertions about them have been checked by a human and/or SageMath.

\subsection*{Statement on conflict of interest}
On behalf of all authors, the corresponding author states that there is no conflict of interest. 

\subsection*{Data availability statement}
Data sharing is not applicable to this article as no datasets were generated or analyzed during the current study.

\section{Preliminaries on notions from graph theory}\label{sec:notions}
 Given a graph $X$, we denote the set of vertices by $V(X)$ and the set of edges by $E(X)$.

\begin{defn}
Let $X$ be a graph. An edge of $X$ is called a \textbf{\textit{bridge}} if its deletion increases the number of connected components. A graph $X$ is said to be \textbf{\textit{bridgeless}} if it contains no bridges.
\end{defn}

\begin{defn}
Let $X$ be a graph with $n$ vertices and denote its vertices by $v_1,\dots, v_n$.
We define the Laplacian matrix of $X$ as $L_X=D_X - A_X$, where $D_X=\Diag(\deg(v_1),\dots,\deg(v_n))$ is the diagonal matrix whose diagonal entries are the degrees of the vertices of $X$ and $A_X$ is the adjacency matrix of $X$, whose $(i,j)$-entry is given by the number of edges connecting $v_i$ and $v_j$.

We define $\Div(X)$ as the free abelian group generated by $V(X)$. Elements of $\Div(X)$ will be called \textbf{\textit{divisors}} of $X$. The degree of a divisor $D=\sum_{i=1}^n a_i v_i \in \Div(X)$ is defined as $\deg(D)=\sum_{i=1}^n a_i$. We denote the set of divisors of degree $k$ by $\Div^{(k)}(X)$. A divisor $E=\sum_{i=1}^n a_i v_i$ is called \textbf{\textit{effective}} if $a_i \ge 0$ for all $i$ and we will write $E\ge0$. We write $\Div_+^{(k)}(X)$ for the set of effective divisors of degree $k$.

The Laplacian matrix defines a $\mathbb{Z}$-linear map
$L_X:\Div(X)\to\Div(X)$ by
\[
L_X(v_j)=\sum_{i=1}^n \ell_{ij}v_i,
\]
where $\ell_{ij}$ is the $(i,j)$-entry of $L_X$.
We define the \textbf{\textit{Picard group}} of X as the quotient group $$\Pic(X)=\frac{\Div(X)}{\im(L_X)},$$ where $\im(L_X)$, which we will call the \textbf{\textit{Laplacian lattice}} of $X$, denotes the image of the Laplacian matrix on $\Div(X)$. We will write $A\sim B$ ($A$ is equivalent to $B$) when the divisors $A$ and $B$ have the same image in the Picard group. The set of effective divisors equivalent to $D$ is denoted by $|D|$. Given $D\in\Div(X)$, we write $[D]$ for its image in $\Pic(X)$.

Since $\im(L_X)$ contains only degree 0 divisors, {if $A\sim B$, then $\deg(A)=\deg(B)$. We can define the degree of an equivalence class $\mathcal{D}\in\Pic(X)$ as $\deg(D)$, where $D\in\Div(X)$ is any representative of $\mathcal{D}$. We write $\Pic^{k}(X)$ for the set of elements} in $\Pic(X)$ of degree $k$. We will call $J_X:=\Pic^0(X)$ the \textbf{\textit{Jacobian}} of $X$ (note that this is a group). 
\end{defn}

\begin{rem}\label{rk:equivalence}
    It is clear from the definition that if $D$ and $D'$ are two divisors of $X$ such that $D\sim D'$, then $|D|=|D'|$. Given a class $[D]\in \Pic(X)$, we {define $|[D]|$ as $|D|$}.
\end{rem}

\begin{rem}
When $X$ is connected, the order of the group $J_X$ is equal to the number of spanning trees of $X$.
\end{rem}

We recall the following definition of \textbf{\textit{Riemann--Roch structure}} from \cite{BN07}.
\begin{defn} Let $X$ be a graph. The Riemann--Roch structure on $X$ is defined as $h: \Div(X) \to \{0, 1, 2,...\}$, where
\[
h(D) = \min\left\{\deg(E):E\in\Div_+(X), |D-E|= \emptyset \right\}.
\]

When $X$ is connected, we define its \textbf{\textit{genus}} as $g{:=} |E(X)|-|V(X)| +1$, and its \textbf{\textit{canonical divisor}} as the divisor $K{:=}\sum\limits_{v\in V(X)} (\deg(v)-2)v\in \Div(X)$.
\end{defn}

Note that $h(D)=0$ if and only if $D$ is not equivalent to an effective divisor, and, if $D_1\sim D_2$, then $h(D_1)=h(D_2)$. In other words, $h$ is a well-defined function on $\Pic(X)$, satisfying $h([D])=h(D)$.

\begin{thm}[Riemann--Roch for graphs]\label{thm:RR} Let $X$ be a connected graph. Then, for all $D\in\Div(X)$,
    \[
    h(D) - h(K-D) = \deg(D)- g+1.
    \]
\end{thm}

\begin{proof}
    See \cite[Theorem~1.12]{BN07}.
\end{proof}

We recall the following result from \cite[Proposition~2.3]{zini00}.
\begin{prop}\label{cor:inj}
Let $X$ be a connected graph.    The natural map $i: V(X) \to \Pic(X)$ is injective if and only if $X$ is bridgeless.
\end{prop}

\begin{rem}\label{rk:inj}
 It follows from \cref{cor:inj} that when $X$ is connected and bridgeless, we can identify $V(X)$ as a subset of $\Pic^1(X)$, and we will do so implicitly in the remainder of the article.    
\end{rem}

\section{Two-variable zeta functions and $L$-functions}\label{sec:functions}
Throughout this section, $X$ is a fixed connected graph. We first recall the definition of Lorenzini's two-variable zeta function associated with a graph given in \cite{zini12}:
\begin{defn}
     We define the two-variable Lorenzini zeta function of $X$ as follows,

\[
\zeta(t,u,X) := \sum_{D\in \Pic(X)} \frac{u^{h(D)}-1}{u-1}t^{\deg(D)}.
\]
\end{defn}

\begin{rem}
Note that $\zeta(t,u,X)$ can be defined by only using the lattice generated by $L_X$. Thus, if two graphs have the same Laplacian lattice, they have the same Lorenzini zeta function.    
\end{rem}
We recall the following description of the Lorenzini zeta function given in \cite[Proposition 3.10]{zini12}. We outline the proof, as this will subsequently facilitate the computation of the examples presented later.

\begin{prop} \label{thm:rat}
We have
\[
\zeta(t,u,X) = \frac{f(t,u)}{(1-t)(1-ut)},
\]
where $f(t,u)\in \mathbb{Z}[t,u]$. Moreover, the degree of $f(t,u)$ in $t$ is $2g$, and  $f(1,u)=|J_X|$.
\end{prop}
\begin{proof}
    Let $\kappa=|J_X|$ and
    \[
    p(t,u) = \sum_{\substack{D\in\Pic(X) \\ \deg D\le2g-2}} \frac{u^{h(D)}-1}{u-1}t^{\deg(D)}\in\mathbb{Z}[t,u].
    \]
    Theorem~\ref{thm:RR} tells us that $\deg(D)>2g-2$ implies $h(D)=\deg(D)-g+1$. For each integer $d>2g-2$, there are $\kappa$ elements in $\Pic^d(X)$. Thus, we can rewrite the Lorenzini zeta function as
    \begin{align*}
        \zeta(t,u,X) &= \sum_{D\in \Pic(X)} \frac{u^{h(D)}-1}{u-1}t^{\deg(D)} \\
        & =p(t,u)+ \sum_{\substack{D\in\Pic(X) \\ \deg D>2g-2}} \frac{u^{\deg(D)-g+1}-1}{u-1}t^{\deg(D)} \\
        & =p(t,u)+ \kappa\sum_{d=2g-1}^\infty \frac{u^{d-g+1}-1}{u-1}t^d \\
        & = p(t,u) +  \frac{\kappa}{u-1}\left(\frac{u^{g}}{1-ut}-\frac{1}{1-t}\right)t^{2g-1} \\
        & = p(t,u) +  \kappa \frac{u^g-u^gt-1+ut}{(1-ut)(1-t)(u-1)}t^{2g-1}.
    \end{align*}
    The proposition follows from the fact that $(u-1)| (u^g-u^gt-1+ut)$ over $\mathbb{Z}$.
\end{proof}

\begin{eg}\label{eg:g2}
Let $X$ be a graph with $n$ vertices and write $\kappa$ for the cardinality of $J_X$. We recall from \cite[Example 3.8]{zini12} that when $g=0$, we have 
    \[
    \zeta(t,u,X) = \frac{1}{(1-t)(1-ut)}.
    \]
    When $g=1$, we have 
    \[
    \zeta(t,u,X) = \frac{1+(\kappa-(u+1))t+ut^2}{(1-t)(1-ut)}.
    \]
    When $g=2$, and $X$ is connected and bridgeless, we have 
    \[
     \zeta(t,u,X) = \frac{1+(n-u-1)t+(2u-nu+\kappa-n)t^2+u(n-u-1)t^3+u^2t^4}{(1-t)(1-ut)}.
    \]
    Indeed, we have the following equations:
    \begin{align*}
    \sum_{D\in\Pic^0(X)} \frac{u^{h(D)}-1}{u-1}=1, \quad
    \sum_{D\in\Pic^1(X)} \frac{u^{h(D)}-1}{u-1}=n , \quad
    \sum_{D\in\Pic^2(X)} \frac{u^{h(D)}-1}{u-1}= \kappa+u .
    \end{align*}
 The second equation follows from the fact that no distinct effective divisors of degree 1 are equivalent and the third equation is a consequence of \cref{thm:RR}. This allows us to compute the polynomial $p(t,u)$ in the proof of \cref{thm:rat}, which in turn allows us to compute $\zeta(t,u,X)$.
\end{eg}

\begin{prop} \label{prop:contraction}
 Let $X$ be a graph, and let $X'$ be the graph obtained by contracting all the bridges in $X$. Then
    \[
    \zeta(t,u,X) = \zeta(t,u,X').
    \]    
\end{prop}
\begin{proof}
If $X$ contains no bridge, then the proposition is immediate. Suppose $X$ contains a bridge $e$. It follows from \cite[Lemma~3.5]{kimoto}\footnote{We are grateful to Dino Lorenzini for pointing out this reference to us.} that
       \begin{align*}
    \zeta(t,u,X) 
    &= \zeta(t,u,X/e).
    \end{align*}
     Repeating this process until no bridges remain gives $\zeta(t,u,X) = \zeta(t,u,X')$.    
\end{proof}

Inspired by the definition of the $L$-functions of coverings of curves, we propose the following definition of the $L$-function associated with a character of $J_X$ and an element of $\Pic^1(X)$.

\begin{defn}\label{def:L}
Let  $\chi : J_X \to \mathbb{C}^\times$ be a character and $D_1\in \Pic^1(X)$. We define the $L$-function associated with $\chi$ and $D_1$ as follows:
\[
L(t,u,X,\chi, D_1) := \sum_{D\in \Pic(X)} \frac{u^{h(D)}-1}{u-1}\chi(D-\deg(D)D_1)t^{\deg(D)}.
\]
\end{defn}

\begin{prop}\label{prop:same h}
    Let $D_1,D_1'\in\Pic^1(X)$. Then,
    \[
    L(t,u,X,\chi, D_1) = L(t,u,X,\chi, D_1')
    \]
    for all characters $\chi$ if and only if 
    \[
    h(E+dD_1)=h(E+dD_1')
    \]
    for all $E\in J_X$ and $d\ge 0$.
\end{prop}

\begin{proof}
    We see that $L(t,u,X,\chi, D_1) = L(t,u,X,\chi, D_1')$ if and only if
    \[
    \sum_{d=0}^\infty\sum_{D_0\in J_X} \frac{u^{h(D_0+dD_1)}-1}{u-1}\chi(D_0)t^d=\sum_{d=0}^\infty\sum_{D_0\in J_X} \frac{u^{h(D_0+dD_1')}-1}{u-1}\chi(D_0)t^d.
    \]
    This is equivalent to
    \begin{equation}\label{eq:equivalence}
    \sum_{D_0\in J_X} (u^{h(D_0+dD_1)}-u^{h(D_0+dD_1')})\chi(D_0)=0       
    \end{equation}
    for all $d\ge0$ and all characters $\chi$ of $J_X$.
    
  Now, let $\widehat{J_X}=\operatorname{Hom}(J_X,\mathbb C^\times)$ and 
let $E\in J_X$. We recall that the orthogonality relations for the characters of the finite abelian group $J_X$ give
\begin{equation}\label{eq:character-orthogonality}
\frac{1}{|J_X|}
\sum_{\chi\in\widehat{J_X}}\chi(D_0-E)
=
\begin{cases}
1& D_0=E,\\
0& D_0\ne E.
\end{cases}
\end{equation}

Thus, for any $E\in J_X$, multiplying \eqref{eq:equivalence} by $\chi(-E)$ and summing over $\chi\in\widehat{J_X}$ gives
\[
u^{h(E+dD_1)}=u^{h(E+dD_1')},
\]
thanks to \eqref{eq:character-orthogonality}.
Since $u$ is an indeterminate, this implies
$h(E+dD_1)=h(E+dD_1')$ for every $E\in J_X$ and $d\ge 0$.
The converse is immediate.
\end{proof}

\begin{prop}\label{prop:L-poly}
Let  $\chi : J_X \to \mathbb{C}^\times$ be a non-trivial character and $D_1\in \Pic^1(X)$. We have
\[
L(t,u,X,\chi{,D_1}) =  \sum_{\substack{D\in\Pic(X) \\ \deg D\le2g-2}} \frac{u^{h(D)}-1}{u-1}\chi(D-\deg(D)D_1)t^{\deg(D)}.
\]
It is a polynomial in $t$ and $u$ over $\mathbb{C}$. Moreover, as a polynomial in $t$, it has degree $2g-2$, and its leading coefficient is $\chi([K]-\deg(K) D_1)u^{g-1}$.

If $\chi=\mathbbm{1}$ is the trivial character, we recover the Lorenzini zeta function:
\[
L(t,u,X,\mathbbm{1}{,D_1}) = \zeta(t,u,X).
\]
\end{prop}
\begin{proof} The case where $\chi=\mathbbm{1}$ is clear. Suppose $\chi\ne \mathbbm{1}$.
Theorem~\ref{thm:RR} tells us that $\deg(D)>2g-2$ implies $h(D)=\deg(D)-g+1$. Thus,
    \begin{align*}
        L(t,u,X,\chi{,D_1}) & = \sum_{D\in \Pic(X)} \frac{u^{h(D)}-1}{u-1}\chi(D-\deg(D)D_1)t^{\deg(D)} \\
        & = \sum\limits_{d=0}^\infty \sum\limits_{D_0\in J_X}\frac{u^{h(D_0+dD_1)}-1}{u-1}\chi(D_0)t^d \\
        & = \sum\limits_{d=0}^{2g-2} \sum\limits_{D_0\in J_X}\frac{u^{h(D_0+dD_1)}-1}{u-1}\chi(D_0)t^d +  \sum\limits_{d=2g-1}^{\infty} \frac{u^{d-g+1}-1}{u-1}t^d\sum\limits_{D_0\in J_X}\chi(D_0)\\
        & = \sum\limits_{d=0}^{2g-2} \sum\limits_{D_0\in J_X}\frac{u^{h(D_0+dD_1)}-1}{u-1}\chi(D_0)t^d \\
        & = \sum_{\substack{D\in\Pic(X) \\ \deg D\le2g-2}} \frac{u^{h(D)}-1}{u-1}\chi(D-\deg(D)D_1)t^{\deg(D)},
    \end{align*}
    as desired. 
    
    To show it has degree $2g-2$ in $t$, we compute its leading coefficient. By \cref{thm:RR}, for all $D\in\Pic^{2g-2}(X)$, $h(D)=g-1$, except for $D=[K]$, in which case $h(D)=g$. Thus,
    \begin{align*}
        &\ \sum_{D\in \Pic^{2g-2}(X)} \frac{u^{h(D)}-1}{u-1}\chi(D-\deg(D)D_1) \\
        = &\ \sum_{D\in \Pic^{2g-2}(X)\backslash\{[K]\}} \frac{u^{g-1}-1}{u-1}\chi(D-\deg(D)D_1)+\frac{u^g-1}{u-1}\chi([K]-\deg(K)D_1) \\
        = &\ -\frac{u^{g-1}-1}{u-1}\chi([K]-\deg (K) D_1)+\frac{u^g-1}{u-1}\chi([K]-\deg(K)D_1) \\
        =&\ \chi([K]-\deg(K)D_1) u^{g-1},
    \end{align*}
    which is non-zero.
\end{proof}

We recall from \cite[Proposition~3.10(a)]{zini12} that the Lorenzini zeta function satisfies a functional equation. The proposition below shows that the $L$-functions satisfy a similar functional equation.\footnote{We thank Dino Lorenzini for suggesting that we study this.}

\begin{prop} Let $D_1\in \Pic^1(X)$, and let $\chi:J_X\to\mathbb{C}^\times$ be a character.  The following functional equation holds
    \[
    L(t,u,X,\chi{,D_1})=\chi([K]-\deg( K)D_1)(ut^2)^{g-1}L\left(\frac{1}{ut},u,X,\chi^{-1}{,D_1}\right).
    \]
\end{prop}

\begin{proof}If $\chi$ is trivial, the functional equation is given by \cite[Proposition~3.10(a)]{zini12}. Suppose $\chi$ is non-trivial and let 
\[
g(t,u{,D_1})=\chi([K]-\deg(K)D_1)(ut^2)^{g-1}L\left(\frac{1}{ut},u,X,\chi^{-1}{,D_1}\right).
\]
It follows from \cref{prop:L-poly} that
\[
g(t,u{,D_1})=\chi([K]-\deg(K)D_1) \sum_{\substack{D\in\Pic(X) \\ \deg D\le2g-2}} \frac{u^{h(D)}-1}{u-1}\chi^{-1}(D-\deg(D)D_1)t^{2g-2-\deg(D)}u^{g-1-\deg(D)}.
\]
Furthermore, \cref{thm:RR} tells us that $h([K]-D)=h(D)-\deg (D)+g-1$ for all $D\in\Pic(X)$. Therefore,
\begin{align*}
g(t,u{,D_1})=&\sum_{\substack{D\in\Pic(X) \\ 0\le\deg D\le2g-2}} \frac{u^{h(D)}-1}{u-1}\chi([K]-D-\deg([K]-D)D_1)t^{2g-2-\deg(D)}u^{g-1-\deg(D)}\\
=&\sum_{\substack{D\in\Pic(X) \\ 0\le\deg D\le2g-2}} \frac{u^{h([K]-D)}-u^{g-1-\deg(D)}}{u-1}\chi([K]-D-\deg([K]-D)D_1)t^{\deg([K]-D)}\\
=&\sum_{\substack{D\in\Pic(X) \\ 0\le\deg D\le2g-2}} \frac{u^{h([K]-D)}-u^{1-g+\deg([K]-D)}}{u-1}\chi([K]-D-\deg([K]-D)D_1)t^{\deg([K]-D)}.
\end{align*}
The map $D\mapsto [K]-D$ gives a bijection from the set of all divisor classes of degree between $0$ and $2g-2$ to itself. Thus, 
\begin{align*}
g(t,u{,D_1})=\sum_{\substack{D\in\Pic(X) \\ 0\le\deg D\le2g-2}} \frac{u^{h(D)}-u^{1-g+\deg(D)}}{u-1}\chi(D-\deg(D)D_1)t^{\deg(D)}.
\end{align*}
Finally, as $\chi$ is non-trivial, we have $\sum\limits_{D_0\in J_X} \chi(D_0)=0$. Hence,
\begin{align*}
g(t,u{,D_1})&=\sum_{d=0}^{2g-2}\sum_{D_0\in J_X} \frac{u^{h(D_0+dD_1)}-u^{1-g+d}}{u-1}\chi(D_0)t^d\\
&=\sum_{d=0}^{2g-2}\sum_{D_0\in J_X} \frac{u^{h(D_0+dD_1)}-1}{u-1}\chi(D_0)t^d\\
&=L(t,u,X,\chi{,D_1}),
\end{align*}
as desired.
\end{proof}

\begin{eg} \label{eg:non-trivial}Let $\chi$ be a non-trivial character and $D_1\in\Pic^1(X)$. Since the zero divisor is the only effective divisor of degree zero, we have $h(D)>0$ for $D\in\Pic^0(X)$ if and only if $D=[0]$. Moreover, it follows directly from the definition of $h$ that $h([0])=1$. Therefore, when $g=1$, we have
    \[
    L(t,u,X,\chi,D_1) = \sum_{D\in\Pic^0(X)}\frac{u^{h(D)}-1}{u-1}\chi(D)=1+\sum_{D\in\Pic^0(X)\backslash\{[0]\}}\frac{u^{h(D)}-1}{u-1}\chi(D)=1. 
    \]

 Now, suppose that $g=2$ and $X$ is connected and bridgeless.  A class $D\in\Pic^1(X)$ satisfies $h(D)>0$ if and
only if it contains an effective divisor of degree one, and hence
if and only if $D=[v]$ for some $v\in V(X)$. Moreover, we have
\[
h([v])=1.
\]
Indeed, since $v$ is effective, we have $h([v])>0$. If $v'\neq v$ are distinct vertices, then the injectivity of the map $i:V(X)\to \Pic(X)$ given by \cref{cor:inj} implies that $[v-v']\neq [0]$. Since a divisor class of degree zero
has positive $h$ if and only if it is trivial, we have
$h([v-v'])=0$. By the definition of $h$, this implies that $h([v])\le 1$. Therefore, we deduce that $h([v])=1$.

We next determine $h$ on $\Pic^2(X)$. Since $\deg(K)=2$, \cref{thm:RR} gives
\[
h(D)=1+h([K]-D)
\]
for every $D\in\Pic^2(X)$. Since $\deg([K]-D)=0$, we have
\[
h(D)=
\begin{cases}
2,&D=[K],\\
1,&D\neq[K].
\end{cases}
\] 
Hence,
    \begin{align*}
       L(t,u,X,\chi,D_1) &= 1+ \sum_{D\in\Pic^1(X)}\frac{u^{h(D)}-1}{u-1}\chi(D-D_1)t+\sum_{D\in\Pic^2(X)}\frac{u^{h(D)}-1}{u-1}\chi(D-2D_1)t^2\\
       &= 1+ \sum_{v\in V(X)}\chi([v]-D_1)t+\left(\sum_{D\in\Pic^2(X)\backslash\{[K]\}}\chi(D-2D_1)+(u+1)\chi([K]-2D_1)\right)t^2\\
       &=1+\sum_{v\in V(X)} \chi([v]-D_1) t + \chi([K] - 2D_1) ut^2. 
    \end{align*}
    The last equality follows from the fact that $\chi$ is non-trivial and $\Pic^2(X)$ is a torsor for $J_X$, so 
    \[
    \sum_{D\in\Pic^2(X)}\chi(D-2D_1)=0.
    \]
 \end{eg}

\begin{eg} \label{eg:genus2}
    Consider the following graph:
\begin{center}
    \begin{tikzpicture}[
    every node/.style={circle, draw, fill=black, inner sep=0pt},]
    \node (a) at (0,0) [label=below:${v_1}$] {0};
    \node (b) at (2,0) [label=below:${v_2}$] {1};

    \draw (a) to (b);
    \draw (a) to [bend left] (b);
    \draw (a) to [bend right] (b);

    \node[draw=none, fill=none] at (1,-1) {$X$};
    \end{tikzpicture}
\end{center}
    Since $X$ has genus 2, we have seen in \cref{eg:g2} that 
    \[
    \zeta(t,u,X) = \frac{1+t-ut+t^2+ut^3-u^2t^3+u^2t^4}{(1-t)(1-ut)}.
    \]
    There are three distinct classes in $\Pic^1(X)$, namely, {$[v_1],[v_2]$ and $[2v_1-v_2]$}, and, two non-trivial characters of $J_X=\langle {[v_2-v_1]}\rangle\cong \ZZ/3\ZZ$ given by 
    \begin{align*}
        \chi_1:{[v_2-v_1]}&\mapsto\zeta_3,\\ 
        \chi_2:{[v_2-v_1]}&\mapsto\zeta_3^2,
    \end{align*}
    where $\zeta_3$ is a primitive third root of unity. We have the following $L$-functions:
    \begin{align*}
        &L(t,u,X,\chi_1,[v_1]) = L(t,u,X,\chi_2,[v_2]) = 1+(1+\zeta_3)t +\zeta_3ut^2, \\
        &L(t,u,X,\chi_2,[v_1]) = L(t,u,X,\chi_1,[v_2]) = 1+(1+\zeta_3^2)t +\zeta_3^2ut^2, \\
        &L(t,u,X,\chi_1,[2v_1-v_2]) = L(t,u,X,\chi_2,[2v_1-v_2]) = 1-t+ut^2 .
    \end{align*}
In particular, we see that the $L$-functions depend on the choice of $D_1\in\Pic^1(X)$.
    
\end{eg}

\begin{rem} \label{rem:noAF}
Recall from \cite[P.158]{StarkTerras2000} that the Ihara zeta functions satisfy the Artin formalism, that is, if $Y/X$ is a Galois cover (see {\cite[Definition 8]{StarkTerras2000}}), then the Ihara zeta function of $Y$ admits a factorization into a product of Artin--Ihara $L$-functions associated with the representations of $\mathrm{Gal}(Y/X)$. When $g=1$, the $L$-functions for the non-trivial characters are all equal to $1$ by \cref{eg:non-trivial}, so the product of all the $L$-functions for any given $D_1\in\Pic^1(X)$ is simply $\zeta(t,u,X)$. When $g\ge2$, the polynomial $f(t,u)$ that appears in the statement of \cref{thm:rat} is irreducible as an element of $\mathbb{C}[t,u]$ by \cite[Proposition~3.10(c)]{zini12}. Thus, for $g\geq 2$, the product over all characters cannot be the Lorenzini zeta function of a connected graph.

One might consider instead the product \[P(t,u,X,D_1):=\prod\limits_{\chi\ne\mathbbm{1}} L(t,u,X, \chi,D_1)\]
for $D_1\in\Pic^1(X)$.
 Let $N$ denote the exponent of $J_X$. It follows from \cref{prop:L-poly} that each $L(t,u,X,\chi,D_1)$ belongs to
$\mathbb{Z}[\mu_N][t,u]$. In particular, the coefficients of
$P(t,u,X,D_1)$ are algebraic integers. For
every $\sigma\in\mathrm{Gal}(\mathbb{Q}(\mu_N)/\mathbb{Q})$, we have
\[
    \sigma\bigl(L(t,u,X,\chi,D_1)\bigr)
    =
    L(t,u,X,\sigma\circ\chi,D_1).
\]
Since the Galois action permutes the non-trivial characters of
$J_X$, it follows that
\[
    \sigma\bigl(P(t,u,X,D_1)\bigr)=P(t,u,X,D_1).
\]
Hence, $P(t,u,X,D_1)\in\mathbb{Z}[t,u]$.

For the graph presented in \cref{eg:genus2}, we see that not only do the $L$-functions depend on the choice of $D_1\in\Pic^1(X)$, but the product $ P(t,u,X, D_1)$ also depends on this choice since
    \begin{align*}
    &P(t,u,X,[v_1])=P(t,u,X,[v_2])=1+t+(1-u)t^2+ut^3+u^2t^4,\\
    &P(t,u,X, [2v_1-v_2])=(1-t+ut^2)^2.
    \end{align*}
    \end{rem}

\section{Recovering Laplacian lattices from $L$-functions}\label{sec:mainproof}
We prove \cref{thmA} in this section. We begin with the following preliminary lemmas. 

\begin{lem}\label{extend-iso}
Let $X$ and $X'$ be graphs such that there exists a group isomorphism
    $\phi:J_{X'}\to J_X$, and let $D_1\in\Pic^1(X)$, $D_1'\in\Pic^1(X')$. For every $D\in\Pic(X')$, set  $D_0=D-\deg(D)D_1'$. Then $D_0\in J_{X'}$.
     Let 
    \begin{align*}
    \tilde\phi:\Pic(X') &\to \Pic(X),\\
    D&\mapsto\phi (D_0) + \deg(D)D_1.    
    \end{align*}
     The map $\tilde\phi$ extends $\phi$ to an isomorphism from $\Pic(X')$ to $\Pic(X)$. {Furthermore, $\tilde\phi(D_1')=D_1$.}
\end{lem}
\begin{proof} 
As $\deg(D_1')=1$, it is clear that for any $D\in\Pic(X')$, we have $\deg(D_0)=0$, and so $D_0\in J_{X'}$. Furthermore, if $D\in J_{X'}$, then $\deg(D)=0$. Thus, $D_0=D$ and $\tilde\phi(D)=\phi(D)$.

Given  $A, B\in \Pic(X')$, with 
\begin{align*}
    A= A_0+\deg(A)D_1',\\
B=B_0+\deg(B)D_1',
\end{align*}
where $A_0,B_0\in J_{X'}$, we have
\begin{align*}
\tilde\phi(A+B) &= \phi(A_0+B_0)+(\deg(A)+\deg(B))D_1\\
                &=\phi(A_0)+\deg(A)D_1 + \phi(B_0)+\deg(B)D_1 \\
                &= \tilde\phi(A)+\tilde\phi(B).
\end{align*}
Thus, $\tilde\phi$ is a group homomorphism.

 Also, $\tilde\phi$ is an isomorphism because it is bijective, with its inverse given by 
 \begin{align*}
     \Pic(X)&\to\Pic(X')\\
    D &\mapsto \phi^{-1}(D_0)+\deg(D)D_1',
 \end{align*}
where $D_0=D-\deg(D)D_1\in J_X$.

When $D=D_1'$, the corresponding $D_0$ is given by $D_1'-\deg(D_1')D_1'=0$. Thus, $\tilde\phi(D_1')=\phi(0)+D_1=D_1$.
\end{proof}

\begin{lem} \label{lem:same}
Let $X$ and $X'$ be connected bridgeless graphs such that there exists a group isomorphism $\phi:J_{X'}\to J_X$. 
 Let $D_1\in\Pic^1(X)$ and $D_1'\in\Pic^1(X')$, and let $\tilde\phi:\Pic(X')\to\Pic(X)$ be the isomorphism extending $\phi$ given as in \cref{extend-iso}. Moreover, assume that $\tilde\phi$ induces a bijection from $V(X')$ to $V(X)$ after regarding $V(X')$ and $V(X)$ as subsets of $\Pic^1(X')$ and $\Pic^1(X)$, respectively, as indicated in \cref{rk:inj}.
 Under this bijection, the natural identification of $\mathbb Z^{V(X')}$ with $\mathbb Z^{V(X)}$ maps the Laplacian lattice of $X'$ bijectively onto that of $X$.
\end{lem}
\begin{proof}
By definition, a divisor $\sum a_i v_i' \in\Div(X')$ belongs to $\im(L_{X'})$ if and only if $\sum a_i v_i'\sim 0$. As $\tilde\phi $ is a group isomorphism, this is equivalent to $\sum a_i \tilde\phi(v_i') \sim 0$, which in turn is equivalent to $\sum a_i \tilde\phi(v_i') \in \im(L_X)$. Thus, if we choose to order the vertices of $X$ by 
\[
v_i = \tilde\phi(v_i'),
\]
we see that $\sum a_i v_i' \in \im(L_{X'})$ if and only if $\sum a_i v_i \in \im(L_{X})$. 
\end{proof}

We are now ready to prove the main theorem of this article.
\begin{thm}\label{mainthm}\label{thmA}
 Let $X$ and $X'$ be connected graphs without bridges. Suppose there exists a group isomorphism $\phi:J_{X'}\to J_X$, and there exist $D_1\in \Pic^1(X)$ and $D_1'\in \Pic^1(X')$ such that for all characters $\chi$ of $J_X$, we have the equality
\[
L(t,u,X,\chi{,D_1}) = L(t,u,X',\chi\circ\phi{,D_1'}).
\]
Then, the isomorphism
$\tilde\phi:\Pic(X')\to\Pic(X)$ extending $\phi$ given as in \cref{extend-iso} satisfies
\[
h_{X'}(D)=h_X\bigl(\tilde\phi(D)\bigr)
\qquad\text{for every }D\in\Pic(X')
\]
and gives rise to a bijection $\hat\phi:V(X')\to V(X)$.
Under $\hat\phi$, the natural identification of
$\mathbb Z^{V(X')}$ with $\mathbb Z^{V(X)}$ maps the
Laplacian lattice of $X'$ bijectively onto that of $X$.\end{thm} 
\begin{proof}
   Let $h_X$ and $h_{X'}$ denote the Riemann--Roch structures for $X$ and $X'$, respectively. Let $\tilde\phi:\Pic(X')\to\Pic(X)$ be the isomorphism extending $\phi$ given by \cref{extend-iso}. We show that for all $v\in V(X')$ (regarded as a subset of $\Pic^1(X')$), the map $\tilde\phi$ sends $v$ to an element of $V(X)\subset\Pic^1(X)$. Notice that 
\begin{align*}
    L(t,u,X,\chi{,D_1}) & = \sum_{D\in \Pic(X)} \frac{u^{h_{{X}}(D)}-1}{u-1}\chi(D-\deg(D)D_1)t^{\deg(D)} \\
                  & = \sum_{D_0\in J_X} \chi(D_0)\sum_{d\ge 0} \frac{u^{h_{{X}}(D_0+dD_1)}-1}{u-1}t^d ,    
\end{align*}
and
\begin{align*}
    L(t,u,X',\chi\circ\phi{,D_1'}) &= \sum_{{D_0'}\in J_{X'}} \chi(\phi ({D_0'}))\sum_{d\ge 0} \frac{u^{h_{{X'}}({D_0'}+dD_1')}-1}{u-1}t^d \\
    & = \sum_{D_0\in J_X} \chi(D_0)\sum_{d\ge 0} \frac{u^{h_{{X'}}(\phi^{-1}(D_0)+dD_1')}-1}{u-1}t^d.
\end{align*}
Therefore, for all $d\ge0$ and characters $\chi$, we have
\[
 \sum_{D_0\in J_X} \chi(D_0)\left(u^{h_{{X'}}(\phi^{-1}(D_0)+dD_1')}-u^{h_{{X}}(D_0+dD_1)}\right)=0.
\]
By the argument in the proof of \cref{prop:same h}, we deduce that for all $d\ge 0$ and $D_0\in J_X$,
\[
h_{{X}}(D_0+dD_1)=h_{{X'}}(\phi^{-1}(D_0)+dD_1').
\]
Since both
Riemann--Roch functions vanish on classes of negative degree,
we obtain
\[
h_{X'}(D)=h_X(\tilde\phi(D))
\qquad\text{for every }D\in\Pic(X').
\]

Thus, given $v\in V(X')\subset \Pic(X')$, we have that $ h_X(\tilde\phi (v))=h_{X'}(v)\ge1$. So, $\tilde\phi(v)$ is equivalent to an effective divisor of degree one, i.e. $\tilde\phi(v)\in V(X)$.

By symmetry,  for all $w\in V(X)\subset\Pic^1(X)$, the map $\tilde\phi^{-1}$ sends $w$ to an element of $V(X')\subset\Pic^1(X')$. Therefore, $\tilde\phi$ induces a bijection from $V(X')$ to $V(X)$. Hence, it follows from \cref{lem:same} that the natural identification of $\mathbb Z^{V(X')}$ with $\mathbb Z^{V(X)}$ maps the Laplacian lattice of $X'$ bijectively onto that of $X$, as desired.
\end{proof}

\begin{cor}\label{cor:iso}
    Let $X$ and $X'$ be connected bridgeless graphs with the property that every pair of distinct vertices within the same graph is connected by at least one edge. Let  $D_1\in \Pic^1(X)$ and $D_1'\in \Pic^1(X')$. Let $\phi:J_{X'}\to J_X$
be a group isomorphism such that for all characters $\chi$ of $J_X$, the $L$-functions defined using $D_1$ and $D_1'$ satisfy
\[
L(t,u,X,\chi{,D_1}) = L(t,u,X',\chi\circ\phi{,D_1'}).
\]
Then, $X$ and $X'$ are isomorphic.
    
\end{cor}
\begin{proof}
By \cite[Corollary~5]{manjunath}, a graph where every pair of distinct vertices is connected by at least one edge is completely determined by its Laplacian lattice. Hence, the corollary follows immediately from \cref{mainthm}.
\end{proof}

\begin{eg}\label{eg:non-iso}
    Consider the following bridgeless graphs.
\begin{center}    
\begin{tikzpicture}[
    vertex/.style={circle, draw, fill=white, minimum size=7mm, inner sep=0pt, font=\small},
    edge/.style={thick},scale=0.7
]

\begin{scope}
    \node[vertex] (0a) at (-1.5, 1.5) {$v_1$};
    \node[vertex] (1a) at (-1.5,-1.5) {$v_2$};
    \node[vertex] (4a) at ( 0.5, 0.0) {$v_5$};
    \node[vertex] (2a) at ( 2.5, 1.2) {$v_3$};
    \node[vertex] (3a) at ( 2.5,-1.2) {$v_4$};

    \draw[edge] (4a) -- (2a);
    \draw[edge] (4a) -- (3a);
    \draw[edge] (2a) -- (3a);

    \draw[edge] (0a) to[bend left=15]  (4a);
    \draw[edge] (0a) to[bend right=15] (4a);

    \draw[edge] (1a) to[bend left=15]  (4a);
    \draw[edge] (1a) to[bend right=15] (4a);

    \node at (0.7,-2.3) {$X$};
\end{scope}

\begin{scope}[xshift=8cm]
    \node[vertex] (0b) at (-3.0, 0.0) {$v_1$};
    \node[vertex] (1b) at (-1.0, 0.0) {$v_2$};
    \node[vertex] (4b) at ( 1.0, 0.0) {$v_5$};
    \node[vertex] (2b) at ( 3.0, 1.2) {$v_3$};
    \node[vertex] (3b) at ( 3.0,-1.2) {$v_4$};

    \draw[edge] (4b) -- (2b);
    \draw[edge] (4b) -- (3b);
    \draw[edge] (2b) -- (3b);

    \draw[edge] (1b) to[bend left=15]  (4b);
    \draw[edge] (1b) to[bend right=15] (4b);

    \draw[edge] (0b) to[bend left=15]  (1b);
    \draw[edge] (0b) to[bend right=15] (1b);

    \node at (0.0,-2.3) {$X'$};
\end{scope}

\end{tikzpicture}

\end{center}
It is clear that $X$ and $X'$ are not isomorphic since they do not share the same degree sequence. Nonetheless, under the identification of vertices as given by the labels above, they have the same Laplacian lattice since their Laplacian matrices have the same {Hermite} normal form.
\end{eg}

\begin{rem}\label{rk:genuine}
Note that an isomorphism \(J_X \cong J_{X'}\) alone is not sufficient to guarantee equality of the corresponding lattices. For example, consider the following graphs:\\

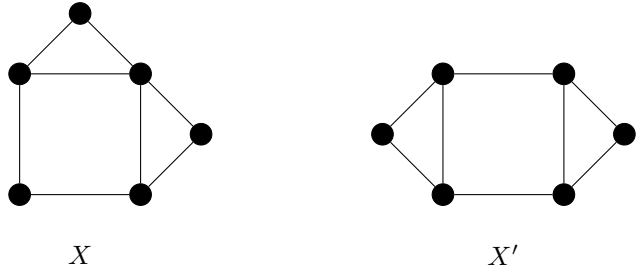
\begin{center}
\begin{tikzpicture}[
    every node/.style={circle, draw, fill=black, inner sep=0pt},scale=0.8
]

\begin{scope}
    \node (a0) at (0,0) {0};
    \node (a1) at (2,0) {1};
    \node (a2) at (2,2) {2};
    \node (a3) at (0,2) {3};

    \node (a4) at (1,3) {4};
    \node (a5) at (3,1) {5};

    \draw (a0)--(a1)--(a2)--(a3)--(a0);
    \draw (a2)--(a4)--(a3);
    \draw (a1)--(a5)--(a2);
    
    \node[draw=none, fill=none] at (1,-1) {$X$};
\end{scope}

\begin{scope}[xshift=7cm]
    \node (b0) at (0,0) {0};
    \node (b1) at (2,0) {1};
    \node (b2) at (2,2) {2};
    \node (b3) at (0,2) {3};

    \node (b4) at (-1,1) {4};
    \node (b5) at (3,1) {5};

    \draw (b0)--(b1)--(b2)--(b3)--(b0);
    \draw (b0)--(b4)--(b3);
    \draw (b1)--(b5)--(b2);
    
    \node[draw=none, fill=none] at (1,-1) {$X'$};

\end{scope}

\end{tikzpicture}
\end{center}

We can explicitly compute the Smith normal forms of their corresponding Laplacians and find that they are equal. We find $J_X\cong J_{X'}\cong\ZZ/30 \ZZ$. However, they have the following distinct Lorenzini zeta functions:
\begin{align*}
\zeta(t,u,X)&=\frac{u^3t^6 - u^3t^5 + 5u^2t^5 - 5u^2t^4 + 12ut^4 - 6ut^3 - 5ut^2 + 12t^3 - ut + 12t^2 + 5t + 1}{(1-t)(1-ut)}, \\
\zeta(t,u,X') &= \frac{u^3t^6 - u^3t^5 + 5u^2t^5 - 4u^2t^4 - u^2t^3 + 12ut^4 - 7ut^3 - 4ut^2 + 12t^3 - ut + 12t^2 + 5t + 1}{(1-t)(1-ut)},
\end{align*}
 which means that the two graphs have distinct Laplacian lattices.
\end{rem}

 \begin{rem}\label{rk:sameLorenzini}
The equality of Lorenzini zeta functions alone does not suffice in \cref{mainthm}, which means that the $L$-functions with non-trivial characters encode some additional information of the graph not included in the Lorenzini zeta function. For example, consider the following graphs: \\
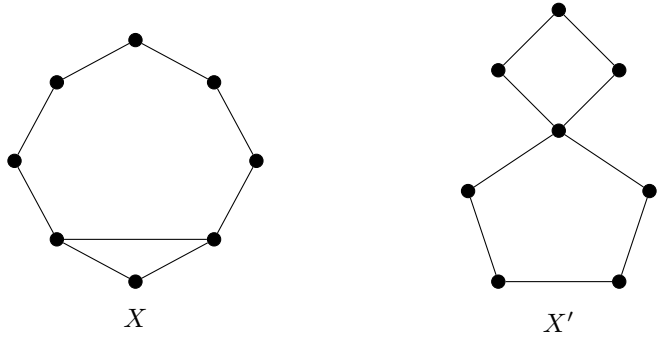
\begin{center}
\begin{tikzpicture}[
    every node/.style={circle, draw, fill=black, inner sep=0pt, minimum size=5pt},scale=0.8
]

\begin{scope}
    \node (0) at (-1.3,-0.3) {};
    \node (1) at (1.3,-0.3) {};
    \node (2) at (0,-1) {};
    \node (3) at (-2,1) {};
    \node (4) at (-1.3,2.3) {};
    \node (5) at (0,3) {};
    \node (6) at (1.3,2.3) {};
    \node (7) at (2,1) {};
    
    \draw
        (0)--(1)
        (1)--(2)
        (2)--(0)
        (0)--(3)
        (3)--(4)
        (4)--(5)
        (5)--(6)
        (6)--(7)
        (7)--(1);

    \node[draw=none, fill=none, below=8pt] at (0,-1)
        {$X$};
\end{scope}
\begin{scope}[xshift=7cm]

    \node (0) at (0,1.5) {};
    \node (1) at (-1,2.5) {};
    \node (2) at (0,3.5) {};
    \node (3) at (1,2.5) {};
    \node (4) at (-1.5,0.5) {};
    \node (5) at (-1, -1) {};
    \node (6) at (1,-1) {};
    \node (7) at (1.5,0.5) {};

    \draw
        (0)--(1)
        (1)--(2)
        (2)--(3)
        (3)--(0)
        (0)--(4)
        (4)--(5)
        (5)--(6)
        (6)--(7)
        (7)--(0);

    \node[draw=none, fill=none, below=8pt] at (0,-1)
        {$X'$};
\end{scope}

\end{tikzpicture}
\end{center}

These are two genus $2$ graphs with $8$ vertices and $20$ spanning trees. By \cref{eg:g2}, this implies that they admit the same Lorenzini zeta function. Moreover, by computing the Smith normal forms, we find that $J_X\cong J_{X'}\cong \mathbb{Z}/20\mathbb{Z}$. However, by performing an exhaustive (brute-force) computation in SageMath, we verify that there exists no permutation matrix \(P\) such that \(\operatorname{im}(L_X) = \operatorname{im}(P^{\mathsf T} L_{X'} P)\); we do so by computing and comparing the Hermite normal forms of the corresponding matrices.
 \end{rem}

\bibliographystyle{alpha}
\bibliography{references.bib}

@article {zini12,
    AUTHOR = {Lorenzini, Dino},
     TITLE = {Two-variable zeta-functions on graphs and {R}iemann-{R}och
              theorems},
   JOURNAL = {Int. Math. Res. Not. IMRN},
  FJOURNAL = {International Mathematics Research Notices. IMRN},
      YEAR = {2012},
    NUMBER = {22},
     PAGES = {5100--5131},
     
}

@article {BV,
    AUTHOR = {Booher, Jeremy and Voloch, Jos\'e{} Felipe},
     TITLE = {Recovering algebraic curves from {$L$}-functions of {H}ilbert
              class fields},
   JOURNAL = {Res. Number Theory},
  FJOURNAL = {Research in Number Theory},
    VOLUME = {6},
      YEAR = {2020},
    NUMBER = {4},
     PAGES = {Paper No. 43, 6},
     
}

@article {zini00,
    AUTHOR = {Lorenzini, Dino},
     TITLE = {Arithmetical properties of {L}aplacians of graphs},
   JOURNAL = {Linear and Multilinear Algebra},
  FJOURNAL = {Linear and Multilinear Algebra},
    VOLUME = {47},
      YEAR = {2000},
    NUMBER = {4},
     PAGES = {281--306},
}

@article{BN07,
   AUTHOR = {Baker, Matthew and Norine, Serguei},
     TITLE = {Riemann-{R}och and {A}bel-{J}acobi theory on a finite graph},
   JOURNAL = {Adv. Math.},
  FJOURNAL = {Advances in Mathematics},
    VOLUME = {215},
      YEAR = {2007},
    NUMBER = {2},
     PAGES = {766--788},
}

@article {kimoto,
    AUTHOR = {Kimoto, Kazufumi},
     TITLE = {Remarks and examples on two-variable zeta functions for
              graphs},
   JOURNAL = {Ryukyu Math. J.},
  FJOURNAL = {Ryukyu Mathematical Journal},
    VOLUME = {36},
      YEAR = {2023},
     PAGES = {49--68},
     
}

@article {Ihara,
    AUTHOR = {Ihara, Yasutaka},
     TITLE = {On discrete subgroups of the two by two projective linear
              group over {$p$}-adic fields},
   JOURNAL = {J. Math. Soc. Japan},
  FJOURNAL = {Journal of the Mathematical Society of Japan},
    VOLUME = {18},
      YEAR = {1966},
     PAGES = {219--235},
     
}

@article{StarkTerras2000,
    AUTHOR = {Stark, Harold and Terras, Audrey},
     TITLE = {Zeta functions of finite graphs and coverings. {II}},
   JOURNAL = {Adv. Math.},
  FJOURNAL = {Advances in Mathematics},
    VOLUME = {154},
      YEAR = {2000},
    NUMBER = {1},
     PAGES = {132--195},
     
}

@article {manjunath,
    AUTHOR = {Manjunath, Madhusudan},
     TITLE = {The {L}aplacian lattice of a graph under a simplicial distance
              function},
   JOURNAL = {European J. Combin.},
  FJOURNAL = {European Journal of Combinatorics},
    VOLUME = {34},
      YEAR = {2013},
    NUMBER = {6},
     PAGES = {1051--1070},
     
}

\end{document}